\documentclass[11pt]{amsart}
\usepackage [T1]{fontenc}
\usepackage [latin1]{inputenc}
\usepackage{amssymb,amsmath,amsthm}
\usepackage{verbatim}
\usepackage{graphicx}
\usepackage{epsfig}

\usepackage{color, soul}
\usepackage[all]{xy}

\newtheorem{theorem}{Theorem}[section]
\newtheorem{definition}[theorem]{Definition}

\newtheorem{lemma}[theorem]{Lemma}
\newtheorem{proposition}[theorem]{Proposition}
\newtheorem{corollary}[theorem]{Corollary}

\def\i{\mathfrak{i}}
\def\j{\mathfrak{j}}
\def\U{\mathfrak{U}}
\def\Q{\mathcal{Q}}
\def\F{\mathfrak{F}}
\def\K{\kappa}
\def\P{\mathcal{P}^n}

\def\hbu{H_{b\Q}}

\begin{document}

\title{Extending polynomials in maximal and minimal ideals}

\thanks{The first author was partially supported by ANPCyT PICT 05 17-33042, UBACyT Grant X038 and ANPCyT PICT 06 00587. The second author was partially supported by ANPCyT PICT 05 17-33042, UBACyT Grant X863 and a Doctoral fellowship from CONICET}

\author{Daniel Carando}

\author{Daniel Galicer}

\address{Departamento de Matem\'{a}tica - Pab I,
Facultad de Cs. Exactas y Naturales, Universidad de Buenos Aires,
(1428) Buenos Aires, Argentina, and CONICET.}
\email{dcarando@dm.uba.ar} \email{dgalicer@dm.uba.ar}

\begin{abstract}
Given an homogeneous polynomial on a Banach space $E$ belonging to some maximal or minimal polynomial ideal, we consider its iterated extension to an ultrapower of $E$ and prove that this extension  remains in the ideal and has the same ideal norm. As a consequence, we show that the Aron-Berner extension is a well defined isometry for any maximal or minimal ideal of homogeneous polynomials. This allow us to obtain symmetric versions of some basic results of the metric theory of tensor products.
\end{abstract}

\subjclass[2000]{46G25, 46A32, 46B28, 47H60}
\keywords{Extension of polynomials, polynomial ideals, symmetric tensor products of Banach spaces}

\maketitle

\section{Introduction}

Aron and Berner showed in~\cite{AroBer78} how to extend continuous polynomials (and some holomorphic functions) defined on a Banach space $E$ to the bidual $E''$. Some time later, Davie and Gamelin~\cite{DavieGamelin98} proved that this extension preserves the norm. This fact is crucial to show that some holomorphic functions defined on the unit ball of $E$ can be extended to the ball of $E''$. A natural question is whether a polynomial ideal is closed under the Aron-Berner extension and, also,  if the ideal norm is preserved by this extension. This is easy for nuclear and approximable polynomials, and it is also known to hold, for example, for integral polynomials~\cite{CarandoZalduendo99}, for extendible polynomials~\cite{Carando99}, and for the ideal of polynomials that are weakly continuous on bounded sets \cite{Moraes84} among others. However, some polynomial ideals are not closed under Aron-Berner extension (for example, the ideal of weakly sequentially continuous polynomials).

Floret and Hunfeld showed that there is an extension, the so called uniterated Aron-Berner extension, which is an isometry for maximal polynomial ideals \cite{Flo02(Ultrastability)}.
Although it is easy to prove properties of this uniterated extension, it is hard to compute, since its definition depends on an ultrafilter. On the other hand, the (iterated) Aron-Berner extension is not only easier to compute, but also has a simple characterization that allows to check if a given extension of a polynomial is actually its Aron-Berner extension~\cite{Zalduedo90}. Moreover, this extension is more widely used in the study of polynomials and analytic functions (for example, it allows a description of the spectrum of the algebra of bounded type analytic functions \cite{AroGalGarMae96}).

Extensions of polynomials to ultrapowers were first studied by Lindst\"{o}m and Ryan in \cite{Lind_Ryan92} and by Dineen and Timoney in \cite{DineenTimoney92}, where they show that different extensions preserve the uniform norm.
In \cite{Flo02(Ultrastability)} the authors further  developed some of this ideas and proved that maximality and ultrastability are equivalent for a polynomial ideal \cite[Theorem 3.2.]{Flo02(Ultrastability)}. Therefore, for maximal ideals, the uniterated extension to the ultrapower turns out to be an isometry.


One of the aims of this note is to prove that the Aron-Berner extension is a well defined isometry on maximal and on minimal ideals. 
More generally, we study the extension of a polynomial on a Banach space $E$ to some ultrapower $(E)_{\U}$, and show the following:
if $\Q$ is a maximal or a minimal ideal of
$n$-homogeneous polynomials and $P \in \Q(E)$, then the
iterated extension $\overline{P}$ of $P$ to the ultrapower
$(E)_\U$ is in $\Q((E)_\U)$, and  $\|P\|_{\Q(E)}=\|\overline{P}\|_{\Q((E)_\U)}.$
As a contribution to the metric theory of symmetric tensor products, we derive the symmetric versions of the Extension Lemma and the Embedding Lemma \cite[13.2 and 13.3]{DefFlo93} (which in the non-symmetric setting are much simpler to prove). These are Corollaries \ref{AB} and \ref{embedding} below.

The article is organized as follows. In section~\ref{prel} we describe same basic properties of ultrapowers and define the extensions mentioned above.
In section~\ref{seccion principal} we prove, following \cite{DavieGamelin98} and using the representation theorem for maximal polynomial ideals, the results stated in the previous paragraph.

\smallskip

We refer to \cite{Flo97,Flo01,Flo01(extension),Flo02(On-ideals),Flo02(Ultrastability)} for the theory of symmetric tensor products and polynomial ideals.

\section{Preliminaries} \label{prel}

Throughout the paper $E$ and $F$ will be Banach spaces, $E'$ the dual space of $E$, $\K_E: E \longrightarrow E''$ the canonical embedding and $B_E$ the open unit ball of $E$. We will denote by $FIN(E)$ the class of all finite dimensional subspaces of the Banach space $E$.

We now recall some basic properties of ultrapowers. The reader is referred to \cite{Heinrich80,Kursten} for further details.
Let $\U$ be an ultrafilter on a set $I$. Whenever the limit with respect to $\U$ of a family $\{ a_\i : \i \in I \}$ exists, we denote it by $\lim_{\i, \U}a_\i$.
For a Banach space $E$, $(E)_{\U}$, the \textit{ultrapower of $E$ respect to the filter $\U$}, consists in classes of elements of the form $z=(z_\i)_\U$, with $z_\i \in E$, for each $\i \in I$, where the norm of $(z_\i)$
is uniformly bounded, and where we identify $(z_\i)$ with $(y_\i)$ if $\lim_{i,\U} \|z_\i - y_\i\| =0$. The space $(E)_{\U}$ is a Banach space
under the norm $$\|(z_\i)_{\U}\| = \lim_{\i, \U} \|z_\i\|.$$

We may consider $E$ as a subspace of the ultrapower $(E)_{\U}$ by means of the canonical embedding
$h_E: E \hookrightarrow (E)_\U$ given by $h_Ex = (x_\i)_\U$ where $x_\i=x$ for all $\i$.

Let us now define the ultrapower of an operator.
If $T : E \to F$ is a bounded linear operator, the \textit{ultrapower operator respect to
the ultrafilter $\U$} will be the operator from $(E)_\U$ to
$(F)_\U$ defined according the following rule $(z_\i)_\U \to
(Tz_\i)_\U$. We will denote this operator $(T)_\U$. It can be
seen that $\|(T)_\U\|$ is $\|T\|$.

We will need a special property of ultrapowers \cite[Proposition 6.1]{Heinrich80}, \cite[Statz 4.1.]{Kursten}:
\begin{proposition}{(Local determination of ultrapowers)} \label{determinacion local} Let $E$ be a Banach space and $M~\in~FIN((E)_\U)$. For each $\i \in I$ there
exist an operator $R_\i \in \mathcal{L}(M,E)$ such that
\begin{enumerate}
\item $z=(R_\i z)_\U$ for all $z \in M$
\item $\|R_\i\| \leq 1$ for all $\i \in I$ and there is an $\mathcal{U} \in \U$ with $\|R_\i\|=1$ for all $\i \in \mathcal{U}$.
\item for all $\varepsilon >0$ there is an
$\mathcal{U}_\varepsilon \in \U$ such that the inverse $R_{\i}^{-1} :
R_\i(M) \to M$ exist and $\|R_\i\| \leq 1 + \varepsilon$ for all
$\i \in \mathcal{U}_\varepsilon$.
\end{enumerate}
\end{proposition}
We shall only use (1) and the first part of (2).

There are different ways of extending polynomials from a Banach space to an ultrapower. Two approaches are the iterated and the uniterated extension.
Let $(E)_\U$ be an ultrapower of a Banach space $E$. For $\Phi$ a continuous $n$-linear function on $E$ we define  an $n$-linear map $\overline{\Phi}$ on $(E)_\U$ by
$$ \overline{\Phi}(z_1,\dots, z_n)=\lim_{\i_1, \U} \dots \lim_{\i_n, \U} \Phi(z_{\i_1}^{(1)}, \dots,   z_{\i_n}^{(n)}),$$
for $z_j=(z_{\i_j}^{(j)})_\U \in (E)_\U$.
If $P$ is an $n$-homogeneous continuous polynomial and $A$ is its associated symmetric $n$-linear mapping,
the \textit{iterated extension}, $\overline{P}$, of $P$ to $(E)_\U$ is defined by
$$ \overline{P}((z_\i)_\U):= \overline{A}((z_\i)_\U, \dots, (z_\i)_\U),$$ and this coincides with $\lim_{\i_1, \U} \dots \lim_{\i_n, \U} A(z_{\i_1}, \dots, z_{\i_n}).$

On the other hand the \textit{uniterated extension}, $\widetilde{P}$, is defined by
$$ \widetilde{P}((z_\i)_\U):= \lim_{\i, \U} P(z_\i).$$

Similarly, there are two analogous ways of extending a polynomial from a Banach space into its bidual. The classical Aron-Berner extension \cite{AroBer78} which is, by construction, iterated (the associated multilinear form is extended from the last variable to the first one by weak-star continuity) and the uniterated Aron-Berner extension (a term coined in  \cite{Flo97}). We need to construct a local ultrapower of $E$ in order to define this uniterated extension, so we give some details.

First, we recall the Principle of Local Reflexivity:
for each $M \in FIN(E'')$, $N \in FIN(E')$ and $\varepsilon > 0$, there
exists an operator $T \in \mathcal{L}(M,E)$ such that
\begin{enumerate}
\item  T is an $\varepsilon$-isometry; that is, $(1-\varepsilon) \|x''\| \leq \|T(x'')\| \leq (1+ \varepsilon ) \|x''\|;$
\item $T(x'')=x''$ for every $x'' \in M \cap E$;
\item $ x'( T (x'') ) = x''(x')$ for $x'' \in M$ and $x' \in
N$.
\end{enumerate}

Let $I$ be the set  of all triples $(M, N, \varepsilon)$, where $M$ and $N$ are finite dimensional subspaces
of $E''$ and $E'$ respectively and $\varepsilon >0$. For each $\i\in I$, we denote by $M_\i$, $N_\i$ and $\varepsilon_\i$ the three elements of the triple.
We define an ordering on $I$ by setting $\i < \j$ if $M_\i \subset M_\j$, $N_\i \subset N_\j$ and $\varepsilon_\i > \varepsilon_\j$.
The collection of the set of the form $B_\i = \{ \j \in I : \i \leq \j \}$ form a filter base. Let $\U$ be an ultrafilter on $I$ which contains this filter base.
The filter $\U$ constructed here is called \textit{a local ultrafilter for $E$}, and $(E)_\U$ is called \textit{a local ultrapower of $E$}.

Finally, let us fix, for each $\i\in I$,  an operator $T_\i : M_\i \to E$ in accordance with the Principle of Local Reflexivity. The canonical embedding of $E$ into the ultrapower $(E)_\U$ extends to a canonical embedding $J_E: E'' \to (E)_\U$ defined by $J_E(x'')=(x_\i)$, where $x_\i$ is equal to $T_\i (x'')$ if $x'' \in M_\i$ and $0$ otherwise. In this way,   $J_E(E'')$ is the range of a norm one
projection on $(E)_\U$. This projection $Pr: (E)_\U \to J_E(E'')$ is given by $Pr((x_\i)_\U)=J_E(w^* - \lim_{\i, \U} x_\i)$ (the weak-star limit in $E''$ of the collection $(x_\i)$).

For a polynomial $P \in \P(E)$ its \textit{uniterated Aron-Berner extension} to $E''$ is defined by $\widetilde{P} \circ J \in \P(E'')$.

\medskip

Let us recall some definitions on the theory of Banach polynomial ideals
\cite{Flo02(On-ideals)}. A \emph{Banach ideal of continuous scalar valued
$n$-homogeneous polynomials} is a pair
$(\mathcal{Q},\|\cdot\|_{\mathcal Q})$ such that:
\begin{enumerate}
\item[(i)] $\mathcal{Q}(E)=\mathcal Q \cap \mathcal{
P}^n(E)$ is a linear subspace of $\mathcal{P}^n(E)$ and $\|\cdot\|_{\mathcal Q}$ is a norm which makes the pair
$(\mathcal{Q},\|\cdot\|_{\mathcal Q})$ a Banach space.

\item[(ii)] If $T\in \mathcal{L} (E_1,E)$, $P \in \mathcal{Q}(E)$ then $P\circ T\in \mathcal{Q}(E_1)$ and $$ \|
P\circ T\|_{\mathcal{Q}(E_1)}\le  \|P\|_{\mathcal{Q}(E)} \| T\|^n.$$

\item[(iii)] $z\mapsto z^n$ belongs to $\mathcal{Q}(\mathbb K)$
and has norm 1.
\end{enumerate}

Let $(\mathcal{Q},\|\cdot\|_{\mathcal Q})$ be the Banach ideal of continuous scalar valued
$n$-homogeneous polynomials and, for $P \in \mathcal{
P}^n(E)$, define
$\|P\|_{\Q^{max}(E)}:= \sup \{ \|P|_M\|_{\Q(M)} : M \in FIN(E) \} \in [0, \infty].$
The maximal kernel of $\mathcal{Q}$ is the ideal given by $ \mathcal{Q}^{max} := \{P \in \mathcal{P}^n : \|P\|_{\Q^{max}} < \infty \}$. An ideal $\mathcal{Q}$ is said to be maximal if $\mathcal{Q} \overset{1}{=} \mathcal{Q}^{max}$.

The minimal kernel of $\mathcal{Q}$ is defined as the composition ideal $\mathcal{Q}^{min} := \mathcal{Q} \circ \overline{\mathfrak{F}}$, where $\overline{\mathfrak{F}}$ stands for the ideal of approximable operators. In other words, a polynomial $P$ belongs to $\mathcal{Q}^{min}(E)$ if it admits a factorization
\begin{equation} \label{factorizacion}
\xymatrix{ E  \ar[rr]^{P} \ar[rd]^{T} & & {\mathbb{K}} \\
& {F} \ar[ru]^{Q} & },
\end{equation}
where $F$ is a Banach space, $T : E \to F$ is an
approximable operator and $Q$ is in $\Q(F)$.
The minimal norm is given by $\|P\|_{\Q^{min}} := \inf \{  \|Q\|_{\Q(F)} \|T\|^n \}$, where the infimum runs over all possible factorizations as in~(\ref{factorizacion}). An ideal $\mathcal{Q}$ is said to be minimal if $\mathcal{Q} \overset{1}{=} \mathcal{Q}^{min}$.

For properties about maximal and minimal ideals of homogeneous polynomials and examples see  \cite{Flo02(Ultrastability),Flo01} and the references therein.

Floret and Hunfeld proved the following result \cite[Theorems 3.2 and 3.3.]{Flo02(Ultrastability)}.
\begin{theorem}
Let $\Q$ be a maximal ideal of $n$-homogeneous polynomials, $Q \in \mathcal{P}^n(E)$ and $\U$ an ultrafilter of $E$. Then $P$ belongs to $\Q(E)$ if and only if the uniterated extension $\widetilde{P} \in \Q((E)_\U)$. In this case $\|\widetilde{P} \|_{\Q((E)_\U)}=\|P\|_{\Q(E)}$.

If $\U$ is a local ultrafilter for $E$ we also have that $P$ belongs to $\Q(E)$ if and only if the uniterated Aron-Berner extension $\widetilde{P} \circ J$ belongs to $\Q(E'')$ and moreover $\|\widetilde{P} \circ J\|_{\Q(E'')}=\|P\|_{\Q(E)}$.
\end{theorem}

We will present a similar theorem for the iterated extension to the ultrapower and for the Aron-Berner extension. We will also conclude that the same holds in the case that $\Q$ is a minimal ideal of homogeneous polynomials.

The following proposition is due to Lindstr\"om and Ryan \cite[Proposition 2.1]{Lind_Ryan92}. It states that the Aron-Berner extension can be recovered from the iterated extension to a local ultrapower of $E$ :
\begin{proposition}\label{propolindrayan}
If $(E)_\U$ is a local ultrapower of $E$, then the restriction of $\overline{P}$ to the canonical image of $E''$ in $(E)_\U$ coincides with the Aron-Berner extension of $P$ to $E''$.
\end{proposition}

\section{The results}\label{seccion principal}

Maximal and minimal ideals of homogeneous polynomials are easily seen to be closed under the Aron-Berner extension: just use a multilinear version of the Extension Lemma \cite[13.2]{DefFlo93} (whose proof is identical) and the main result of \cite{Flo01(extension)}. In this section we will show that this extension is actually an isometry.

First, let $A$ be the symmetric multilinear form associated to a polynomial $P$. For each fixed $j$, $1 \leq j \leq n$, $x_1, \dots, x_{j-1} \in
E$, and $z_j, z_{j+1}, \dots z_n \in (E)_\U$, we have
$$\overline{A}(h_Ex_1, \dots, h_Ex_{j-1}, z_j, z_{j+1}, \dots,
z_n) = \lim_{i_j, \U} \overline{A}(h_Ex_1, \dots,
h_Ex_{j-1},h_Ez_{{\i}_j}^{(j)}, z_{j+1}, \dots, z_{n}),$$
where $\overline A$ is the iterated extension of $A$ to a local ultrapower.

Now, we will imitate the procedure used by Davie and
Gamelin in \cite{DavieGamelin98}. Denote $A$ the symmetric $n$-linear form associated to
$P$. We have the following lemma:

\begin{lemma}\label{diferentes indices} Let $M \in FIN((E)_\U)$ and $z_1, \dots, z_r \in M$. For a given natural number $m$, and $\varepsilon >
0$ there exist operators $R_1, \dots, R_m \in \mathcal{L}(M,E)$
with norm less or equal to 1 such that
\begin{equation}\label{desigualdad lema} \big| A( R_{i_1}z_k, \dots,
R_{i_n}z_k ) - \overline{A}(z_k, \dots, z_k)  \big| < \varepsilon \end{equation} for every $i_1, \dots, i_n$
distinct indices between $1$ and $m$ and every $k = 1 \dots r$.
\end{lemma}

\begin{proof}
Since $A$ is symmetric, in order  to prove the Lemma it suffices to obtain \eqref{desigualdad lema}
for $i_1 < \dots < i_n$. We will select the operator
$R_1, \dots, R_m$ inductively by the following procedure: by
Proposition~\ref{determinacion local}, for each $\i \in I$ there
exist an operator $R_\i \in \mathcal{L}(M,E)$ with norm less or
equal to 1 such $z_k = (R_\i z_k)_\U$.

Since $z_k = (R_\i z_k)_\U$ for each $k$,
the set $\{ \i \in I \; : \: \overline{A}(h_E R_\i z_k,z_k, \dots, z_k)
- \overline{A}(z_k,z_k, \dots,z_k) \big| < \varepsilon /n  \}$ belongs to the filter
$\U$. Therefore, we can pick $R_1 \in \mathcal{L}(M,E)$ such that
$$\big| \overline{A}(h_ER_1z_k,z_k, \dots, z_k) -  \overline{A}(z_k,z_k, \dots,z_k) \big| < \varepsilon /n  ,$$ for every $k=1 \dots r$.

In a similar way we can choose $R_2$ such that
$$\big| \overline{A}(h_ER_2z_k,z_k, \dots, z_k) -  \overline{A}(z_k,z_k, \dots,z_k) \big| < \varepsilon /n
,$$and moreover,
$$\big| \overline{A}(h_ER_1z_k, h_ER_2z_k, z_k, \dots, z_k) -  \overline{A}(h_ER_1z_k,z_k, \dots,z_k) \big| <
\varepsilon /n ,$$ for every $k$. Proceeding in this way, we get $R_l$'s so that

$$\big| \overline{A}(h_ER_{i_1}z_k, \dots, h_ER_{i_{r-1}}z_k, h_ER_{i_{r}}z_k, z_k, \dots,  z_k) -   \overline{A}(h_ER_{i_1}z_k, \dots, h_ER_{i_{r-1}}z_k, z_k, \dots,  z_k) \big| <
\varepsilon /n ,$$ whenever $i_1 < \dots <i_r$ and $k=1,\dots r$. Then,
$$ \big| \overline{A}(h_E R_{i_1}z_k, \dots, h_ER_{i_{n}}z_k) - \overline{A}(z_k \dots,
z_k)\big|$$ is estimated  by the sum of $n$ terms
$$ \big| \overline{A}(h_ER_{i_1}z_k, \dots, h_ER_{i_{n}}z_k) - \overline{A}(h_ER_{i_1}z_k, \dots, h_ER_{i_{n-1}}z_k,
z_k) \big| + \; \dots \;  $$ $$ + \big| \overline{A}(h_ER_{i_1}z_k, z_k
\dots, z_k) - \overline{A}(z_k, \dots,z_k) \big|,$$ each smaller
than $\varepsilon /n $, for all $k=1 \dots, r$.
\end{proof}

\begin{proposition} \label{Lemma Aprox}  Let $M \in FIN((E)_\U)$ and $z_1, \dots, z_r \in M$, $P: E \to
\mathbb{K}$ a continuous polynomial and $\varepsilon > 0$. There
exist a finite subset $\F$ of $\mathbb{N}$ and operators $(R_i)_{i \in \F}$ in
$\mathcal{L}(M,E)$ with norm less or equal than 1, verifying that
$$ \big| \sum_{k=1}^r \overline{P} \big( z_k\big) - \sum_{k=1}^r P \big( \frac{1}{|\F|} \sum_{i \in \F} R_i z_k \big) \big| < \varepsilon. $$
\end{proposition}

\begin{proof} For
$\varepsilon > 0$, fix $m$ large enough and choose $R_1, \dots, R_m$ as
in the previous Lemma, such that $$ \big| \overline{A}(z_k, \dots,
z_k) - A( R_{i_1}z_k, \dots, R_{i_n}z_k ) \big| < \varepsilon/{2r} $$
for every $i_1, \dots, i_n$ distinct indices between $1$ and $m$
and every $k = 1 \dots r$. Now, we set $\F = \{1, \dots,m\}$ and define
$R \in \mathcal{L}(M,E)$ given by $R: = \frac{1}{|\F|} \sum_{i \in
\F} R_i$. For  $k \in \{1, \dots, r\}$, we have
\begin{align*}
\big| \overline{P}(z_k) - P(Rz_k) \big| & = \big| \frac{1}{m^n}
\sum_{i_1, \dots, i_n = 1}^m [\overline{A}(z_k, \dots, z_k) -
A(R_{{i_1}}z_k, \dots, R_{{i_n}}z_k )] \big| \\
& \leq \big| \Sigma_1^k \big| + \leq | \Sigma_2^k |,
\end{align*}
where $\Sigma_1^k$ is the sum over the $n$-tuples of non-repeated
indices (which is less than $\varepsilon/{2r}$) and $\Sigma_2^k$ is
the sum over the remaining indices. It is easy to show that there
are exactly $m^n - \prod_{j=0}^{n-1} (m-j)$ summands in
$\Sigma_2^k$, each bounded by a constant $C>0$ (obviously we can
assume that $C$ is independent of $k$), thus
$$\big| \Sigma_2^k  | \leq \frac{1}{m^n} \big (m^n - \prod_{j=0}^{n-1}
(m-j) \big)C= \big[ 1 - (1-\frac{1}{m}) \dots (1 - \frac{n-1}{m})
\big] C.$$ Taking $m$ sufficiently large this is less than
$\varepsilon/{2r}$.
\end{proof}

Recall that an s-tensor norm $\alpha$ is called finitely generated if for every Banach space $E$ and $z \in \otimes^{n,s} E$, we have: $ \alpha (z, \otimes^{n,s}E) = \inf \{ \alpha(z, \otimes^{n,s}M) : M \in FIN(E),\; z \in \otimes^{n,s}M \}.$
Now we can state the main theorem:
\begin{theorem} \label{isom AB maximales}
Let $\alpha$ be a finitely generated s-tensor norm and $P \in \big( \widetilde{\otimes}^{n,s}_\alpha E \big)'$ a
polynomial. The iterated extension $\overline{P}$ of $P$ to the
ultrapower $(E)_\U$ belongs to $\big( \widetilde{\otimes}^{n,s}_\alpha
(E)_\U \big)'$ and
$$ \|P\|_{\big( \widetilde{\otimes}^{n,s}_\alpha E \big)'} = \|\overline{P}\|_{\big( \widetilde{\otimes}^{n,s}_\alpha (E)_\U
\big)'}.$$
Equivalently, if $\Q$ is a maximal ideal of
$n$-homogeneous polynomials and $P \in \Q(E)$,  the
iterated extension $\overline{P}$ of $P$ to the ultrapower
$(E)_\U$ belongs to $\Q((E)_\U)$ and
$$\|P\|_{\Q(E)}=\|\overline{P}\|_{\Q((E)_\U)}.$$
\end{theorem}

\begin{proof} Thanks to the representation theorem for maximal polynomial ideals \cite[Section 3.2]{Flo02(Ultrastability)} (see also \cite[Section 4]{Flo01}), it is enough to show the first statement.

Let $w \in \otimes^{n,s} M$, where $M \in FIN((E)_\U)$. Since $\alpha$ is finitely generated, we only have to show
that $$| \langle \overline{P},w \rangle| \leq \|P\|_{\big(
\widetilde{\otimes}^{n,s}_\alpha E \big)'} \: \alpha(w,
\otimes^{n,s}M).$$ Now, $w = \sum_{k=1}^r \otimes^n z_k$ with $z_k \in
M$. Given $\varepsilon > 0$, by Proposition~\ref{Lemma Aprox} we
can take a finite set $\F$ and operators $(R_i)_{i \in \F}$ with
$\|R_i\|_{\mathcal{L}(M,E)} \leq 1$ such that $ \big|
\sum_{k=1}^r \overline{P} \big( z_k\big) - \sum_{k=1}^r P \big(
\frac{1}{|\F|} \sum_{i \in \F} R_i z_k \big) \big| < \varepsilon $.
Therefore,
\begin{align*}
\big| \langle \overline{P},w \rangle \big|  & = \big| \sum_{k=1}^r \overline{P}
\big( z_k\big) \big| \leq \big| \sum_{k=1}^r \overline{P} \big(
z_k\big) - \sum_{k=1}^r P \big( \frac{1}{|\F|} \sum_{i \in \F} R_i
z_k \big) \big| + \big| \sum_{k=1}^r P \big( \frac{1}{|\F|}
\sum_{ i \in \F} R_i z_k \big) \big| \\
& \leq \varepsilon + \big| \langle P, \sum_{k=1}^r \otimes^n
\frac{1}{|\F|}
\sum_{i \in \F} R_i z_k \rangle \big| \\
& \leq \varepsilon + \|P\|_{\big( \widetilde{\otimes}^{n,s}_\alpha
E \big)'} \alpha( \sum_{k=1}^r \otimes^n \frac{1}{|\F|} \sum_{i \in
\F} R_i z_k \: ; \; \otimes^{n,s}E) \\
& \leq \varepsilon + \|P\|_{\big( \widetilde{\otimes}^{n,s}_\alpha
E \big)'} \alpha( \otimes^{n,s}R (\sum_{k=1}^r  z_k) \:
; \; \otimes^{n,s}E),
\end{align*}
where $R = \frac{1}{|\F|} \sum_{i \in \F} R_i$ (note that
$\|R\|_{\mathcal{L}(M,E)} \leq 1$ since each
$\|R_i\|_{\mathcal{L}(M,E)} \leq 1$). By the metric mapping
property of $\alpha$ and the previous inequality we get
$$ \big| \langle \overline{P},w \rangle  \big| \leq \varepsilon + \|P\|_{\big( \widetilde{\otimes}^{n,s}_\alpha
E \big)'} \alpha( \sum_{k=1}^r \otimes^n z_k \: ; \;
\otimes^{n,s}M),$$ which ends the proof.
\end{proof}

The following result can be seen as a symmetric version of the Extension Lemma \cite[13.2]{DefFlo93}.

\begin{corollary}\label{AB}
Let $\alpha$ be a finitely generated s-tensor norm of order $n$ and $P \in \big( \widetilde{\otimes}^{n,s}_\alpha E \big)'$ be a
polynomial. Then, the Aron-Berner extension $AB(P)$ of $P$ is in
$\big(\widetilde{\otimes}^{n,s}_\alpha E'' \big)'$ and
$$ \|P\|_{\big( \widetilde{\otimes}^{n,s}_\alpha E \big)'} = \|AB(P)\|_{\big( \widetilde{\otimes}^{n,s}_\alpha E''
\big)'}.$$
Therefore, if $\Q$ is a maximal ideal of
$n$-homogeneous polynomials and we take $P \in \Q(E)$, then its
Aron-Berner extension is in $\Q(E'')$ and
$$\|P\|_{\Q(E)}=\|AB(P)\|_{\Q(E'')}.$$
\end{corollary}

\begin{proof}
Let $(E)_\U$ a local ultrapower of $E$ and  $J_E: E'' \to
(E)_\U$ the canonical embedding. By Proposition~\ref{propolindrayan} the iterated
extension to the local ultrapower of $E$ restricted to $E''$
coincides with the Aron-Berner extension of $P$. In other words,
$AB(P) = \overline{P}\circ J_E$. Therefore,
\begin{align*}
\|AB(P)\|_{\Q(E'')} & = \| \overline{P} \circ J_E
\|_{\Q(E'')} \\
& \leq \| \overline{P} \|_{\Q((E)_\U)} \|J_E\|^n \\
& = \| P \|_{\Q(E)}.
\end{align*}
The other inequality is immediate.
\end{proof}

As a direct consequence we also obtain a symmetric version of the Embedding Lemma \cite[13.3]{DefFlo93}:

\begin{corollary}\label{embedding}
The natural maps
\begin{align*}
& \otimes^{n,s} J_E: \otimes^{n,s}_{\alpha} E \longrightarrow \otimes^{n,s}_{\alpha} (E)_\U \\
&  \otimes^{n,s}\K_E: \otimes^{n,s}_{\alpha} E \longrightarrow \otimes^{n,s}_{\alpha} E''
\end{align*}
are isometries for every finitely generated tensor norm $\alpha$.
\end{corollary}
\begin{proof}
Clearly $\|\otimes^{n,s} J_E\| \leq 1$, and $\langle P, w \rangle = \langle \overline{P}, \otimes^{n,s} J_Ew \rangle$ gives the remaining inequality. Similarly we get that $\otimes^{n,s}\K_E$ is an isometry.
\end{proof}

\medskip

Now we turn our attention to minimal ideals of polynomials. In order to show that the Aron-Berner extension is also an isometry for minimal ideals, we need first the following simple result:

\begin{lemma}\label{aprox}
Let $T : E \to F$ be approximable operator. Then, $(T)_\U :(E)_\U \to (F)_\U$ is also approximable.
\end{lemma}
\begin{proof}
It is sufficient to show that, if $T: E \to F$ is a rank-one
operator, $(T)_\U$ also is. Let $x' \in E'$ such that
$T(x)=x'(x)y$. If we denote $\lambda := \lim_{i,\U} x'(x_i)$ we
have easily that $(T)_\U (x_i)_\U = (x'(x_i)y)_\U = \lambda \;
h_Fy$.
\end{proof}

\begin{theorem}\label{isom AB minimales} Let $\Q$ be a minimal ideal. If $P \in {\Q}(E)$, then the iterated extension $\overline{P}$ belongs to ${\Q}((E)_\U)$ and
$$\|P\|_{{\Q}(E)}=\|\overline{P}\|_{{\Q}((E)_\U)}.$$
\end{theorem}

\begin{proof}
Since $P \in {\Q}(E)\overset{1}
=\big(({\Q}^{max})^{min}\big)(E)$ (see
\cite[3.4]{Flo01}), given $\varepsilon > 0$ there exist a Banach space $F$, an
approximable operator $T : E \to F$ and a polynomial $Q \in
{\Q}^{max}(F)$ such that $P = Q\circ T$
%
and $ \|Q\|_{{\Q}^{max}(F)} \|T\|^n \leq
\|P\|_{{\Q}(E)} + \varepsilon$ (as in (\ref{factorizacion})). Notice that
$\overline{P} = \overline{Q} \circ (T)_\U$. By Theorem~\ref{isom
AB maximales} we have
$\|Q\|_{{\Q}^{max}(F)}=\|\overline{Q}\|_{{\Q}^{max}((F)_\U)}$.
Since by Lemma~\ref{aprox}, $(T)_\U$ is also approximable, we conclude:
$$\|\overline{P}\|_{{\Q}((E)_\U)} \leq   \|\overline{Q}\|_{{\Q}^{max}((F)_\U)} \|(T)_\U\|^n =  \|Q\|_{{\Q}^{max}(F)} \|T\|^n \leq \|P\|_{{\Q}(E)} +
\varepsilon.$$ The reverse inequality is immediate.
\end{proof}

Similarly as in Corollary~\ref{AB} we have:

\begin{corollary}\label{isom AB minimales}
Let $\Q$ be a minimal ideal. For $P \in {\Q}(E)$, its Aron-Berner extension
$AB(P)$ belongs to ${\Q}(E'')$ and
$$\|P\|_{{\Q}(E)}=\|AB\|_{{\Q}(E'')}.$$
\end{corollary}

\bigskip
We end this note with some comments on the extension of analytic functions associated to polynomial ideals.

The concept of holomorphy type was introduced by Nachbin in \cite{Nac69}
(see also \cite{Din71(holomorphy-types)}). The most natural holomorphy types can be seen as sequences of polynomial ideals $\Q=\{\Q_k\}_k$ ($\Q_k$ is an ideal of polynomials of degree $k$, $k=1,2,\dots$), where some kind of affinity between ideals of different degrees is necessary \cite{BotBraJunPel06,CarDimMurN}.  In \cite{CarDimMur07}, given such a sequence of polynomial ideals, an associated Fr\'{e}chet space of entire  functions is defined. In~\cite{Muro-tesis}, the corresponding definition for analytic functions defined on the unit ball of a Banach is given:
\begin{definition}
Let $\Q=\{\Q_k\}_k$  be a  sequence of polynomial ideals and $E$ be a
Banach space. The space of $\Q$-holomorphic functions of
bounded type on $B_E$ is defined as
$$
\hbu(B_E)=\left\{f\in H(B_E)\ : \frac{d^kf(0)}{k!} \in \Q_k(E) \textrm{
and } \lim_{k\rightarrow \infty}
\Big\|\frac{d^kf(0)}{k!}\Big\|_{\Q_k(E)}^{1/k}<1 \right\}.
$$
\end{definition}
An example of this kind of spaces is that of boundedly-integral holomorphic functions in the ball $H_{bI}$, studied in~\cite{DimGalMaeZal04}.

An immediate consequence of our results is the following: let $\Q=\{\Q_k\}_k$  be a  sequence of polynomial ideals, each $\Q_k$ being either maximal or minimal. If $E$ is a Banach space, then a holomorphic function  $f$ belongs to $\hbu(B_E)$ if and only if its Aron-Berner extension belongs to $\hbu(B_{E''})$.
Note that no coherence between ideals of different degrees is needed for this to hold.

\end{document}